\documentclass[english]{article}
\usepackage{color}

\usepackage{amsmath}
\usepackage{amsfonts}
\usepackage{amsthm}
\usepackage{amssymb}
\usepackage{bm}
\usepackage[applemac]{inputenc}
\usepackage[mathcal,mathbf]{euler}
\usepackage{enumerate}
\makeatletter

\newtheoremstyle{mystyle}{}{}{\slshape}{2pt}{\scshape}{.}{ }{} 

\newtheorem{thm}{Theorem}[section]
\newtheorem{cor}[thm]{Corollary}

\newtheorem{prop}[thm]{Proposition}
\newtheorem{lemme}[thm]{Lemma}

\newtheorem{por}[thm]{Porism}

\theoremstyle{definition}
\newtheorem{defi}[thm]{Definition}
\theoremstyle{mystyle}

\theoremstyle{remark}
\newtheorem{rem}[thm]{Remark}

\newcommand{\MM}{\mathbb M}

\DeclareMathOperator{\dom}{dom}

\title{On dp-minimal ordered structures}
\author{Pierre Simon}
\begin{document}
\maketitle
\begin{abstract}
We show basic facts about dp-minimal ordered structures. The main results are : dp-minimal groups are abelian-by-finite-exponent, in a divisible ordered dp-minimal group, any infinite set has non-empty interior, and any theory of pure tree is dp-minimal.
\end{abstract}

\section*{Introduction}

One of the latest topic of interest in abstract model theory is the study of dependent, or $NIP$, theories. The abstract general study, was initiated by Shelah in \cite{Sh715}, and pursued by him in \cite{Sh783}, \cite{Sh863} and \cite{Sh900}. One of the questions he addresses is the definition of \emph{super-dependent} as an analog of superstable for stable theories. Although, as he writes, he has not completely succeeded, the notion he defines of strong-dependence seems promising. In \cite{Sh863} it is studied in details and in particular, ranks are defined. Those so-called dp-ranks are used to prove existence of an indiscernible sub-sequence in any long enough sequence. Roughly speaking, a theory is strongly dependent if no type can fork infinitely many times, each forking being independent from the previous one. (Stated this way, it is naturally a definition of ``strong-$NTP_2$"). Also defined in that paper are notions of minimality, corresponding to the ranks being equal to 1 on 1-types. In \cite{Ons}, Onshuus and Usvyatsov extract from this material the notion of dp-minimality which seems to be the relevent one. A dp-minimal theory is a theory where there cannot be two independent witnesses of forking for a 1-type. It is shown in that paper that a stable theory is dp-minimal if and only if every 1-type has weight 1. In general, unstable, theories, one can link dp-minimality to \emph{burden} as defined by H. Adler (\cite{burden}).

Dp-minimality on ordered structures can be viewed as a generalization of weak-o-minimality. In that context, there are two main questions to address : what do definable sets in dimension 1 look like, ({\it i.e.} how far is the theory from being o-minimal), and what theorems about o-minimality go through. J. Goodrick has started to study those questions in \cite{Good}, focussing on groups. He proves that definable functions are piecewise locally monotonous extending a similar result from weak-o-minimality.
\\

In the first section of this paper, we recall the definitions and give equivalent formulations. In the second section, we make a few observations on general linearly ordered inp-minimal theories showing in particular that, in dimension 1, forking is controlled by the ordering. The lack of a cell-decomposition theorem makes it unclear how to generalize results to higher dimensions.

In section 3, we study dp-minimal groups and show that they are abelian-by-finite-exponent. The linearly ordered ones are abelian. We prove also that an infinite definable set in a dp-minimal ordered divisible group has non-empty interior, solving a conjecture of Alf Dolich.

Finally, in section 4, we give examples of dp-minimal theories. We prove that colored linear orders, orders of finite width and trees are dp-minimal.

\section{Preliminaries on dp-minimality}

\begin{defi}
(Shelah) An independence (or inp-) pattern of length $\kappa$ is a sequence of pairs$(\phi^\alpha(x,y),k^\alpha)_{\alpha < \kappa}$ of formulas such that there exists an array $\langle a^\alpha_{i}: \alpha < \kappa, i < \lambda \rangle$ for some $\lambda \geq \omega$ such that :
\begin{itemize}
\item Rows are $k^\alpha$-inconsistent : for each $\alpha<\kappa$, the set $\{\phi^\alpha(x,a^\alpha_i) : i < \lambda\}$ is $k^\alpha$-inconsistent,
\item paths are consistent : for all $\eta \in \lambda^\kappa$, the set $\{\phi^\alpha(x,a^\alpha_{\eta(\alpha)}) : \alpha < \kappa\}$ is consistent.
\end{itemize}
\end{defi}

\begin{defi}
\begin{itemize}
\item (Goodrick) A theory is inp-minimal if there is no inp-pattern of length two in a single free variable $x$.
\item (Onshuus and Usvyatsov) A theory is dp-minimal if it is $NIP$ and inp-minimal.
\end{itemize}
\end{defi}

A theory is $NTP_2$ if there is no inp-pattern of size $\omega$ for which the formulas $\phi^\alpha(x,y)$ in the definition above are all equal to some $\phi(x,y)$. It is proven in \cite{ntp2} that a theory is $NTP_2$ if this holds for formulas $\phi(x,y)$ where $x$ is a single variable. As a consequence, any inp-minimal theory is $NTP_2$.

We now give equivalent definitions (all the ideas are from \cite{Sh863}, we merely adapt the proofs there from the general $NIP$ context to the dp-minimal one).

\begin{defi}
Two sequences $(a_i)_{i \in I}$ and $(b_j)_{j\in J}$ are \emph{mutually indiscernible} if each one is indiscernible over the other.
\end{defi}

\begin{lemme}\label{defs}
Consider the following statements : 
\begin{enumerate}
\item $T$ is inp-minimal.
\item For any two mutually indiscernible sequences $A=(a_i : i< \omega)$, $B=(b_j : j<\omega)$ and any point $c$, one of the sequences $(tp(a_i/c) : i < \omega)$, $(tp(b_i/c) : i < \omega)$ is constant.
\item Same as above, but change the conclusion to : one the sequences $A$ or $B$ stays indiscernible over $c$.
\item For any indiscernible sequence $A=(a_i : i \in I)$ indexed by a dense linear order $I$, and any point $c$, there is $i_0$ in the completion of $I$ such that the two sequences $(tp(a_i/c) : i<i_0)$ and $(tp(a_i/c) : i>i_0)$ are constant.
\item Same as above, but change the conclusion to : the two sequences $(a_i : i<i_0)$ and $(a_i : i>i_0)$ are indiscernible over $c$.
\item $T$ is dp-minimal.
\end{enumerate}
Then for any theory $T$, (2), (3), (4), (5), (6) are equivalent and imply (1). If $T$ is $NIP$, then they are all equivalent.
\end{lemme}
\begin{proof}
(2) $\Rightarrow$ (1) : In the definition of independence pattern, one may assume that the rows are mutually indiscernible. This is enough.
\\

(2) $\Rightarrow$ (3) : Assume $A=\langle a_i : i< \omega\rangle$, $B=\langle b_i : i<\omega\rangle$ and $c$ are a witness to $\neg$(3). Then there are two tuples $(i_1 < … < i_n)$, $(j_1 < … < j_n)$ and a formula $\phi(x;y_1,…,y_n)$ such that $\models \phi(c;a_{i_1},…,a_{i_n}) \wedge \neg \phi(c;a_{j_1},…,a_{j_n})$. Take an $\alpha < \omega$ greater than all the $i_k$ and the $j_k$. Then, exchanging the $i_k$ and $j_k$ if necessary, we may assume that $\models \phi(c;a_{i_1},…,a_{i_n}) \wedge \neg \phi(c;a_{n.\alpha},…,a_{n.\alpha+n-1})$. Define $A'=\langle (a_{i_1},…,a_{i_n}) \rangle ~\hat{ }~\langle (a_{n.k},…,a_{n.k+n-1}) : k \geq \alpha \rangle$. Construct the same way a sequence $B'$. Then $A'$, $B'$, $c$ give a witness of $\neg$(2).
\\

(3) $\Rightarrow$ (2) : Obvious.
\\

(3) $\Rightarrow$ (5) : Let $A=\langle a_i : i \in I\rangle$ be indiscernible and let $c$ be a point. Then assuming (3) holds, for every $i_0$ in the completion of $I$, one of the two sequences $A_{<i_0} = \langle a_i : i<i_0\rangle$ and $A_{>i_0}=\langle a_i : i>i_0\rangle$ must be indiscernible over $c$. Take any such $i_0$ such that both sequences are infinite, and assume for example that $A_{>i_0}$ is indiscernible over $c$. Let $j_0 = inf\{i \leq i_0 : A_{>i}$ is indiscernible over $c$ $\}$. Then $A_{>j_0}$ is indiscernible over $c$. If there are no elements in $I$ smaller than $j_0$, we are done. Otherwise, if $A_{<j_0}$ is not indiscernible over $c$, then one can find $j_1 < j_0$ such that again $A_{<j_1}$ is not indiscernible over $c$. By definition of $j_0$, $A_{>j_1}$ is not indiscernible over $c$ either. This contradicts (3).
\\

(5) $\Rightarrow$ (4) : Obvious.
\\

(4) $\Rightarrow$ (2) : Assume $\neg$ (2). Then one can find a witness of it consisting of two indiscernible sequences $A=\langle a_i : i\in I\rangle$, $B=\langle b_i : i\in I\rangle$ indexed by a dense linear order $I$ and a point $c$.

Now, we can find an $i_0$ in the completion of $I$ such that for any $i_1 < i_0 < i_2$ in $I$, there are $i, i'$, $i_1 < i < i_0 < i' < i_2$ such that 
$tp(a_i/c) \neq tp(a_{i'}/c)$.
Find a similar point $j_0$ for the sequence $B$. Renumbering the sequences if necessary, we may assume that $i_0 \neq j_0$. Then the indiscernible sequence of pairs $\langle(a_i,b_i) : i \in I \rangle$ gives a witness of $\neg$ (4).
\\

(6) $\Rightarrow$ (2) : Let $A$, $B$, $c$ be a witness of $\neg$ (2). Assume for example that there is $\phi(x,y)$ such that $\models \phi(c,a_0) \wedge \neg\phi(c,a_1)$. Then set $A' = \langle (a_{2k},a_{2k+1}) : k<\omega \rangle$ and $\phi'(x;y_1,y_2) = \phi(x;y_1) \wedge \neg\phi(x;y_2)$. Then by $NIP$, the set $\{\phi'(x,\bar y) : \bar y \in A'\}$ is $k$-inconsistent for some $k$. Doing the same construction with $B$ we see that we get an independence pattern of length 2.
\\

(5) $\Rightarrow$ (6) : Statement (5) clearly implies $NIP$ (because IP is always witnessed by a formula $\phi(x,y)$ with $x$ a single variable). We have already seen that it implies inp-minimality.
\end{proof}

Standard examples of dp-minimal theories include :
\begin{itemize}
\item O-minimal or weakly o-minimal theories (recall that a theory is weakly-o-minimal if every definable set in dimension 1 is a finite union of convex sets),
\item C-minimal theories,
\item $Th(\mathbf Z,+,\leq)$.
\end{itemize}
The reader may check this as an exercise or see \cite{Good}.

More examples are given in section 4 of this paper.

\section{Inp-minimal ordered structures}

Little study has been made yet on general dp-minimal ordered structures. We believe however that there are results to be found already at that general level. In fact, we prove here a few lemmas that turn out to be useful for the study of groups.

We show that, in some sense, forking in dimension 1 is controlled by the order.
\\

We consider $(M,<)$ an inp-minimal linearly ordered structure with no first nor last element. We denote by $T$ its theory, and let $\MM$ be a monster model of $T$.

\begin{lemme}\label{fork}
Let $X=X_{\bar a}$ be a definable subset of $\MM$, cofinal in $\MM$. Then $X$ is non-forking (over $\emptyset$).
\end{lemme}
\begin{proof}
If $X_{\bar a}$ divides over $\emptyset$, there exists an indiscernible sequence $(\bar a_i)_{i < \omega}$, $\bar a_0 = \bar a$, witnessing this. Every $X_{\bar a_i}$ is cofinal in $\MM$. Now pick by induction intervals $I_k$, $k<\omega$, with $I_k < I_{k+1}$ containing a point in each $X_{\bar a_i}$.
We obtain an inp-pattern of length 2 by considering $x \in X_{\bar a_i}$ and $x \in I_k$.

If $X_{\bar a}$ forks over $\emptyset$, it implies a disjunction of formulas that divide, but one of these formulas must be cofinal : a contradiction.
\end{proof}

A few variations are possible here. For example, we assumed that $X$ was cofinal in the whole structure $\MM$, but the proofs also works if $X$ is cofinal in a $\emptyset$-definable set $Y$, or even contains an $\emptyset$-definable point in its closure. This leads to the following results.

For $X$ a definable set, let $Conv(X)$ denote the convex hull of $X$. It is again a definable set.
\begin{por}
Let $X$ be a definable set of $\MM$ (in dimension 1). Assume $Conv(X)$ is $A$ definable. Then $X$ is non-forking over $A$.
\end{por}

\begin{por}
Let $M \prec N$ and let $p$ be a complete 1-type over $N$. If the cut of $p$ over $N$ is of the form $+\infty$, $-\infty$, $a^+$ or $a^-$ for $a \in M$, then $p$ is non-forking over $M$.
\end{por}

Proposition \ref{convfk} generalizes this.

\begin{lemme}
Let $X$ be an $A$-definable subset of $\MM$. Assume that $X$ divides over some model $M$, then : \begin{enumerate}
\item We cannot find $(a_i)_{i<\omega}$ in $M$ and points $(x_i)_{i<\omega}$ in $X(\MM)$ such that $a_0 < x_0 < a_1 < x_1 < a_2 < …$.

\item The set $X$ can be written as a finite disjoint union $X = \bigcup X_i$ where the $X_i$ are definable over $M\cup A$, and each $Conv(X_i)$ contains no $M$-point.
\end{enumerate}
\end{lemme}
\begin{proof}
Easy ; (2) follows from (1).
\end{proof}

\begin{prop}\label{convfk}
Let $A \subset M$, with $M$, $|A|^+$-saturated, and let $p \in S_1(M)$. The following are equivalent :
\begin{enumerate}
\item The type $p$ forks over $A$,
\item There exist $a,b \in M$ such that $p\vdash a < x < b$, and $a$ and $b$ have the same type over $A$,
\item There exist $a,b \in M$ such that $p\vdash a<x<b$, and the interval $I_{a,b}= \{x : a<x<b\}$ divides over $A$.
\end{enumerate}
\end{prop}
\begin{proof}
(3) $\Rightarrow$ (1) is trivial.
\\

For (2) $\Rightarrow$ (3), it is enough to show that if $a \equiv_A b$, then $I_{a,b}$ divides over $A$. Let $\sigma$ be an $A$-automorphism sending $a$ to $b$. Then the tuple $(b=\sigma(a),\sigma(b))$ has the same type as $(a,b)$, and $a < b < \sigma(b)$. By iterating, we obtain a sequence $a_1 < a_2 < … $ such that $(a_k,a_{k+1})$ has the same type over $A$ as $(a,b)$. Now the sets $I_{a_{2k},a_{2k+1}}$ are pairwise disjoint and all have the same type over $A$. Therefore each of them divides over $M$.
\\

We now prove (1) $\Rightarrow$ (2)

Assume that (2) fails for $p$. Let $X_{\bar a}$ be an $M$-definable set such that $p\vdash X_{\bar a}$. Let $\bar a_0 = a, \bar a_1,\bar a_2, …$ be an $A$-indiscernible sequence. Note that the cut of $p$ is invariant under all $A$-automorphisms. Therefore each of the $X_{\bar a_i}$ contains a type with the same cut over $M$ as $p$. Now do a similar reasoning as in Lemma \ref{fork}.
\end{proof}

\begin{cor}
Forking equals dividing : for any $A \subset B$, any $p\in S(B)$, $p$ forks over $A$ if and only if $p$ divides over $A$.
\end{cor}
\begin{proof}
By results of Chernikov and Kaplan (\cite{CK}), it is enough to prove that no type forks over its base. And it suffices to prove this for one-types (because of the general fact that if $tp(a/B)$ does not fork over $A$ and $tp(b/Ba)$ does not fork over $Aa$, then $tp(a,b/B)$ does not fork over $A$).

Assume $p\in S_1(A)$ forks over $A$. Then by the previous proposition, $p$ implies a finite disjunction of intervals $\bigcup_{i<n} (a_i,b_i)$ with $a_i \equiv_A b_i$. Assume $n$ is minimal. Without loss, assume $a_0 < a_1 < … $. Now, as $a_0 \equiv_A b_0$ we can find points $a'_i, b'_i$, with $(a_i,b_i) \equiv_A (a'_i,b'_i)$ and $a'_0 = b_0$.

Then $p$ proves $\bigcup_{i<n} (a'_i,b'_i)$. But the interval $(a_0, b_0)$ is disjoint from that union, so $p$ proves $\bigcup_{0<i<n} (a_i,b_i)$, contradicting the minimality of $n$.
\end{proof}

Note that this does not hold without the assumption that the structure is linearly ordered. In fact the standard example of the circle with a predicate $C(x,y,z)$ saying that $y$ is between $x$ and $z$ (see for example \cite{Wag}, 2.2.4.) is dp-minimal.

\begin{lemme}
Let $E$ be a definable equivalence relation on $M$, we consider the imaginary sort $S=M/E$. Then there is on $S$ a definable equivalence relation $\sim$ with finite classes such that there is a definable linear order on $S/\sim$.
\end{lemme}
\begin{proof}
Define a partial order on $S$ by $a/E \prec b/E$ if  $\inf(\{x:xEa\}) < \inf(\{x:xEb\})$. Let $\sim$ be the equivalence relation on $S$ defined by $x \sim y$ if $\neg(x \prec y \vee y\prec x)$. Then $\prec$ defines a linear order on $S/\sim$. The proof that $\sim$ has finite classes is another variation on the proof of \ref{fork}.
\end{proof}

From now until the end of this section, we also assume $NIP$.

\begin{lemme}
($NIP$). Let $p \in S_1(\MM)$ be a type inducing an $M$-definable cut, then $p$ is definable over $M$.
\end{lemme}
\begin{proof}
We know that $p$ does not fork over $M$, so by $NIP$, $p$ is $M$-invariant. Let $M_1$ be an $|M|^+$-saturated model containing $M$. Then the restriction of $p$ to $M_1$ has a unique global extension inducing the same cut as $p$. In particular $p$ has a unique heir. Being $M$-invariant, $p$ is definable over $M$.
\end{proof}

The next lemma states that members of a uniformly definable family of sets define only finitely many ``germs at $+\infty$".
\begin{lemme}\label{germ1}
($NIP$). Let $\phi(x,y)$ be a formula with parameters in some model $M_0$, $x$ a single variable. Then there are $b_1,…,b_n$ such that for every $b$, there is $\alpha \in \MM$ and $k$ such that the sets $\phi(x,b) \wedge x>\alpha$ and $\phi(x,b_k) \wedge x>\alpha$ are equal.
\end{lemme}
\begin{proof}
Let $E$ be the equivalence relation defined on tuples by $b E b'$ iff $(\exists \alpha)(x>\alpha \rightarrow (\phi(x,b) \leftrightarrow \phi(x,b')))$. Let $b,b'$ having the same type over $M_0$. By $NIP$, the formula $\phi(x,b) \triangle \phi(x,b')$ forks over $M_0$. By Lemma \ref{fork}, this formula cannot be cofinal, so $b$ and $b'$ are $E$-equivalent. This proves that $E$ has finitely many classes.
\end{proof}

If the order is dense, then this analysis can be done also locally around a point $a$ with the same proof :
\begin{lemme}\label{germ2}
($NIP$ + dense order). Let $\phi(x,y)$ be a formula with parameters in some model $M_0$, $x$ a single variable. Then there exists $n$ such that : For any point $a$, there are $b_1,…,b_n$ such that for all $b$, there is $\alpha < a < \beta$ and $k$ such that the sets $\phi(x,b) \wedge \alpha<x<\beta$ and $\phi(x,b_k) \wedge \alpha<x<\beta$ are equal.
\end{lemme}

\section{Dp-minimal groups}

We study inp-minimal groups. Note that by an example of Simonetta, (\cite{Sim}), not all such groups are abelian-by-finite. It is proven in \cite{Cmin} that C-minimal groups are abelian-by-torsion. We generalize the statement here to all inp-minimal theories.

\begin{prop}\label{inpgrp}
Let $G$ be an inp-minimal group. Then there is  a definable normal abelian subgroup $H$ such that $G/H$ is of finite exponent.\end{prop}
\begin{proof}
Let $A, B$ be two definable subgroups of $G$. If $a \in A$ and $b \in B$, then there is $n > 0$ such that either $a^n \in B$ or $b^n \in A$. To see this, assume $a^n \notin B$ and $b^n \notin A$ for
all $n > 0$. Then, for $n \neq m$, the cosets $a^m B$ and $a^n B$ are distinct, as are $A.b^m$ and $A.b^n$. Now we obtain an independence pattern of length two by considering the sequences of formulas $\phi_k(x)=``x \in a^k B$" and $\psi_k(x)=``x \in A.b^k$".

For $x \in G$, let $C(x)$ be the centralizer of $x$. By compactness, there is $k$ such that for $x,y \in G$, for some $k' \leq k$, either $x^{k'} \in C(y)$ or $y^{k'} \in C(x)$. In particular, letting $n = k!$, $x^n$ and $y^n$ commute.

Let $H = C(C(G^n))$, the bicommutant of the $n$th powers of $G$. It is an abelian definable subgroup of $G$ and for all $x \in G$, $x^n \in H$. Finally, if $H$ contains all $n$ powers then it is also the case of all conjugates of $H$, so replacing $H$ by the intersection of its conjugates, we obtain what we want.
 \end{proof}

Now we work with ordered groups.

\begin{lemme}\label{findex}
Let $G$ be an inp-minimal ordered group. Let $H$ be a definable sub-group of $G$ and let $C$ be the convex hull of $H$. Then $H$ is of finite index in $C$.
\end{lemme}
\begin{proof}
We may assume that $H$ and $C$ are $\emptyset$-definable. So without loss, assume $C=G$.

If $H$ is not of finite index, there is a coset of $H$ that forks over $\emptyset$. All cosets of $H$ are cofinal in $G$. This contradicts Lemma \ref{fork}.
\end{proof}

\begin{prop}
Let $G$ be an inp-minimal ordered group, then $G$ is abelian.
\end{prop}
\begin{proof}
Note that if $a,b \in G$ are such that $a^n = b^n$, then $a=b$, for if for example $0<a<b$, then $a^n < a^{n-1}b < a^{n-2}b^2 < … < b^n$.

For $x\in G$, let $C(x)$ be the centralizer of $x$. We let also $D(x)$ be the convex hull of $C(x)$. By \ref{findex}, $C(x)$ is of finite index in $D(x)$. Now take $x\in G$ and $y\in D(x)$. Then $xy$ is in $D(x)$, so there is $n$ such that $(xy)^{n} \in C(x)$. Therefore $(yx)^{n} = x^{-1} (xy)^{n} x = (xy)^{n}$. So $xy=yx$ and $y \in C(x)$. Thus $C(x)=D(x)$ is convex.

Now if $0 <x<y \in G$, then $C(y)$ is a convex subgroup containing $y$, so it contains $x$, and $x$ and $y$ commute.
\end{proof}

This answers a question of Goodrick (\cite{Good} 1.1).
\\

Now, we assume NIP, so $G$ is a dp-minimal ordered group. We denote by $G^+$ the set of positive elements of $G$.

Let $\phi(x)$ be a definable set (with parameters). For $\alpha \in G$, define $X_\alpha = \{g \in G^+ : (\forall x > \alpha) (\phi(x) \leftrightarrow \phi(x+g))\}$. Let $H_\alpha$ be equal to $X_\alpha \cup -X_\alpha \cup \{0\}$. Then $H_\alpha$ is a definable subgroup of $G$ and if $\alpha < \beta$, $H_\alpha$ is contained in $H_\beta$. Finally, let $H$ be the union of the $H_\alpha$ for $\alpha \in G$, it is the subgroup of \emph{eventual periods} of $\phi(x)$.

Now apply Lemma \ref{germ1} to the formula $\psi(x,y) = \phi(x-y)$. It gives $n$ points $b_1,…,b_n$ such that for all $b \in G$, there is $k$ such that $b-b_k$ is in $H$. This implies that $H$ has finite index in $G$.
\\

If furthermore $G$ is densely ordered, then we can do the same analysis locally. This yields a proof of a conjecture of Alf Dolich : in a dp-minimal divisible ordered group, any infinite set has non empty interior. As a consequence, a dp-minimal divisible definably complete ordered group is o-minimal.

As before, $I_{a,b}$ denotes the open interval $(a,b)$, and $\tau_b$ is the translation by $-b$. We will make use of two lemmas from \cite{Good} that we recall here for convenience.
\begin{lemme}[\cite{Good}, 3.3]\label{dense}
Let $G$ be a divisible ordered inp-minimal group, then any infinite definable set is dense in some non trivial interval.
\end{lemme}

In the following lemma, $\overline M$ stands for the completion of $M$. By a definable function $f$ into $\overline M$, we mean a function of the form $a \mapsto \inf \phi(a;M)$ where $\phi(x;y)$ is a definable function. So one can view $\overline M$ as a collection of imaginary sorts (in which case it naturally contains only \emph{definable} cuts of $M$), or understand $f : M \rightarrow \overline M$ simply as a notation.

\begin{lemme}[\cite{Good}, 3.19]\label{bdd}
Let $f : M \rightarrow \overline M$ be a definable partial function such that $f(x)>0$ for all $x$ in the domain of $f$. Then for every interval $I$, there is a sub-interval $J \subseteq I$ and $\epsilon >0$ such that for $x\in J \cap \dom(f)$, $|f(x)|\geq \epsilon$.
\end{lemme}

\begin{thm}\label{interior}
Let $G$ be a divisible ordered dp-minimal group. Let $X$ be an infinite definable set, then $X$ has non-empty interior.
\end{thm}
\begin{proof}
Let $\phi(x)$ be a formula defining $X$.

By Lemma \ref{dense}, there is an interval $I$ such that $X$ is dense in $I$. By Lemma \ref{germ2} applied to $\psi(x;y)=\phi(y+x)$ at 0, there are $b_1,…,b_n \in M$ such that for all $b \in M$, there is $\alpha >0$ and $k$ such that $|x|<\alpha \rightarrow (\phi(b+x) \leftrightarrow \phi(b_k+x))$.

Taking a smaller $I$ and $X$, if necessary, assume that for all $b \in I \cap X$, we may take $k=1$.

Define $f : x \mapsto \sup \{y : I_{-y,y} \cap \tau_{b_1} X = I_{-y,y} \cap\tau_x X\}$, it is a function into $\overline M$, the completion of $M$. By Lemma \ref{bdd}, there is $J \subset I$ such that, for all $b \in J$, we have $|f(b)|\geq \epsilon$.

Fix $\nu < \frac \epsilon 2$ and $b \in J$ such that $I_{b-2\epsilon,b+2\epsilon} \subseteq J$ (taking smaller $\epsilon$ if necessary). Set $L = I_{b-\nu,b+\nu}$ and $Z = L \cap X$. Assume for simplicity $b=0$. Easily, if $g_1,g_2 \in Z$, then $g_1+g_2 \in Z \cup L^c$ and $-g_1 \in Z$ (because any two points of $Z$ have isomorphic neighborhoods of size $\epsilon$). So $Z$ is a group interval : it is the intersection with $I_{b-\nu,b+\nu}$ of some subgroup $H$ of $G$. Now if $x,y \in L$ satisfy that there is $\alpha >0$ such that $I_{-\alpha,\alpha}\cap \tau_x X = I_{-\alpha,\alpha} \cap \tau_y X$, then $x\equiv y$ modulo $H$. It follows that points of $L$ lie in finitely many cosets modulo $H$. Assume $Z$ is not convex, and take $g \in  L \setminus Z$. Then for each $n \in \mathbf N$, the point $g/n$ is in $L$ and the points $g/n$ define infinitely many different cosets; a contradiction.

Therefore $Z$ is convex and $X$ contains a non trivial interval.
\end{proof}

\begin{cor}
Let $G$ be a dp-minimal ordered group. Assume $G$ is divisible and definably complete, then $G$ is o-minimal.
\end{cor}
\begin{proof}
Let $X$ be a definable subset of $G$. By \ref{interior}, the (topological) border $Y$ of $X$ is finite.

Let $a \in X$, then the largest convex set in $X$ containing $a$ is definable. By definable completeness, it is an interval and its end-points must lie in $Y$.

This shows that $G$ is o-minimal.
\end{proof}

\section{Examples of dp-minimal theories}

We give examples of dp-minimal theories, namely : linear orders, order of finite width and trees.

We first look at linear orders. We consider structures of the form $(M,\leq,C_i,R_j)$ where $\leq$ defines a linear order on $M$, the $C_i$ are unary predicates (``colors"), the $R_j$ are binary monotone relations (that is $x_1 \leq x R_j y \leq y_1$ implies $x_1 R_j y_1$).

The following is a (weak) generalization of Rubin's theorem on linear orders (see \cite{Poizat}).

\begin{prop}
Let $(M,\leq,C_i,R_j)$be a colored linear order with monotone relations. Assume that all $\emptyset$-definable sets in dimension 1 are coded by a color and all monotone $\emptyset$-definable binary relations are represented by one of the $R_j$. Then the structure eliminates quantifiers.
\end{prop}
\begin{proof}
The result is obvious if $M$ is finite, so we may assume (for convenience) that this is not the case.

We prove the theorem by back-and-forth. Assume that $M$ is $\omega$-saturated and take two tuples $\bar x = (x_1,…,x_n)$ and $\bar y = (y_1,…,y_n)$ from $M$ having the same quantifier free type.

Take $x_0 \in M$; we look for a corresponding $y_0$. Notice that $\leq$ is itself a monotone relation, a finite boolean combinations of colors is again a color, a positive combination of monotone relations is again a monotone relation, and if $xRy$ is monotone $\phi(x,y)=\neg yRx$ is monotone.
By compactness, it is enough to find a $y_0$ satisfying some finite part of the quantifier-free type of $x_0$; that is, we are given
\begin{itemize}
\item One color $C$ such that $M \models C(x_0)$,
\item For each $k$, monotone relations $R_k$ and $S_k$ such that $M \models x_0 R_k x_k \wedge x_k S_k x_0$.
\end{itemize}
\vspace{3pt}
Define $U_k(x) = \{t : t R_k x_k\} $ and $V_k(x) = \{t : x S_k t\}$. The $U_k(x)$ are initial segments of $M$ and the $V_k(x)$ final segments. For each $k,k'$, either $U_k(x_k)\subseteq U_{k'}(x_{k'})$ or $U_{k'}(x_{k'}) \subseteq U_k(x_k)$. Assume for example $U_k(x_k) \subseteq U_{k'}(x_{k'})$, then this translates into a relation $\phi(x_k,x_{k'})$, where $\phi(x,y) = (\forall t)(t R_k x \rightarrow t R_{k'} y)$. Now $\phi(x,y)$ is a monotone relation itself. The assumptions on $\bar x$ and $\bar y$ therefore imply that also $U_k(y_k) \subseteq U_{k'}(y_{k'})$.

The same remarks hold for the final segments $V_k$.

Now, we may assume that $U_1(x_1)$ is minimal in the $U_k(x_k)$ and $V_l(x_l)$ is minimal in the $V_k(x_k)$. We only need to find a point $y_0$ satisfying $C(x)$ in the intersection $U_1(y_1) \cap V_l(y_l)$.

Let $\psi(x,y)$ be the relation $(\exists t)(C(t) \wedge t R_1 y \wedge x R_l t)$. This is a monotone relation. As it holds for $(x_0,x_l)$, it must also hold for $(y_0,y_l)$, and we are done.
\end{proof}

The following result was suggested, in the case of pure linear orders, by John Goodrick.

\begin{prop}\label{ordre}
Let $\mathcal M = (M,\leq,C_i,R_j)$ be a linearly ordered infinite structure with colors and monotone relations. Then $Th(\mathcal M)$ is dp-minimal.
\end{prop}
\begin{proof}
By the previous result, we may assume that $T= Th(\mathcal M)$ eliminates quantifiers. Let $(x_i)_{i \in I}$, $(y_i)_{i \in I}$ be mutually indiscernible sequences of $n$-tuples, and let $\alpha \in M$ be a point. We want to show that one of the following holds : 
\begin{itemize}
\item For all $i,i' \in I$, $x_i$ and $x_{i'}$ have the same type over $\alpha$, or
\item for all $i,i' \in I$, $y_i$ and $y_{i'}$ have the same type over $\alpha$.
\end{itemize}
Assume that $I$ is dense without end points.

By quantifier elimination, we may assume that $n=1$, that is the $x_i$ and $y_i$ are points of $M$. Without loss, the $(x_i)$ and $(y_i)$ form increasing sequences. Assume there exists $i<j \in I$ and $R$ a monotone definable relation such that $M \models \neg \alpha R x_i \wedge \alpha R x_j$. By monotonicity of $R$, there is a point $i_R$ of the completion of $I$ such that $i < i_R \rightarrow \neg \alpha R x_i$ and $i > i_R \rightarrow \alpha R x_i$.

Assume there is also a monotone relation $S$ and an $i_S$ such that $i < i_S \rightarrow \neg \alpha S y_i$ and $i > i_S \rightarrow \alpha S y_i$.

For points $x,y$ define $I(x,y)$ as the set of $t \in M$ such that $M \models \neg t R x \wedge t Ry$. This is an interval of $M$. Furthermore, if $i_1 < i_2 < i_3 < i_4$ are in $I$, then the intervals $I(x_{i_1},x_{i_2})$ and $I(x_{i_3},x_{i_4})$ are disjoint. Define $J(x,y)$ the same way using $S$ instead of $R$.

Take $i_0 < i_R < i_1 < i_2 < …$ and $j_0 < i_S < j_1 < j_2 < …$. For $k< \omega$, define $I_k = I(x_{i_{2k}},x_{i_{2k+1}})$ and $J_k = J(y_{j_{2k}},y_{j_{2k+1}})$. The two sequences $(I_k)$ and $(J_k)$ are mutually indiscernible sequences of disjoint intervals. Furthermore, we have $\alpha \in I_0 \wedge J_0$. By mutual indiscernibility,  $I_i \wedge J_j \neq \emptyset$ for all indices $i$ and $j$, which is impossible.
\\

We treated the case when $\alpha$ was to the left of the increasing relations $R$ and $S$. The other cases are similar.
\end{proof}

An ordered set $(M,\leq)$ is of \emph{finite width}, if there is $n$ such that $M$ has no antichain of size $n$.

\begin{cor}
Let $\mathcal M=(M,\leq)$ be an infinite ordered set of finite width, then $Th(\mathcal M)$ is dp-minimal.
\end{cor}
\begin{proof}
We can define such a structure in a linear order with monotone relations : see \cite{Shm}. More precisely, there exists a structure $P=(P,\prec,R_j)$ in which $\prec$ is a linear order and the $R_j$ are monotone relations. There is a definable relation $O(x,y)$ such that the structure $(P,O)$ is isomorphic to $(M,\leq)$.

The result therefore follows from the previous one.
\end{proof}

We now move to trees. A tree is a structure $(T,\leq)$ such that $\leq$ defines a partial order on $T$, and for all $x \in T$, the set of points smaller than $x$ is linearly ordered by $\leq$. We will also assume that given $x,y \in T$, the set of points smaller than $x$ and $y$ has a maximal element $x\wedge y$ (and set $x \wedge x = x$). This is not actually a restriction, since we could always work in an imaginary sort to ensure this.

Given $a,b \in T$, we define the open ball $B(a;b)$ of center $a$ containing $b$ as the set $\{x \in T : x\wedge b > a\}$, and the closed ball of center $a$ as $\{x \in T : x \geq a\}$.

Notice that two balls are either disjoint or one is included in the other.

\begin{lemme}\label{orth}
Let $(T,\leq)$ be a tree, $a \in T$, and let $D$ denote the closed ball of center $a$. Let $\bar x =(x^1,…,x ^n) \in (T\setminus D)^n$ and $\bar y=(y^1,…,y^m) \in D^m$. Then $tp(\bar x/a) \cup tp(\bar y/a)\vdash tp(\bar x \cup \bar y/a)$.
\end{lemme}
\begin{proof}
A straightforward back-and-forth, noticing that $tp(\bar x/a) \cup tp(\bar y/a)\vdash tp_{qf}(\bar x \cup \bar y/a)$ (quantifier-free type).
\end{proof}

We now work in the language $\{\leq, \wedge\}$, so a sub-structure is a subset closed under $\wedge$.
\begin{prop}\label{typear}
Let $A=(a_1,…,a_n)$, $B=(b_1,…,b_n)$ be two sub-structures from $T$. Assume :
\begin{enumerate}
\item $A$ and $B$ are isomorphic as sub-structures,
\item for all $i,j$ such that $a_i \geq a_j$, $tp(a_i,a_j) = tp(b_i,b_j)$.
\end{enumerate}
Then $tp(A)=tp(B)$.
\end{prop}
\begin{proof}
We do a back-and-forth. Assume $\mathcal T$ is $\omega$-saturated and $A$, $B$ satisfy the hypothesis. We want to add a point $a$ to $A$. We may assume that $A \cup \{a\}$ forms a sub-structure (otherwise, if some $a_i \wedge a$ is not in $A \cup \{a\}$, add first this element).

We consider different cases :
\begin{enumerate}
\item The point $a$ is below all points of $A$. Without loss $a_0$ is the minimal element of $A$ (which exists because $A$ is closed under $\wedge$). Then find a $b$ such that $tp(a_0,a)=tp(b_0,b)$. For any index $i$, we have : $tp(a_i,a_0)=tp(b_i,b_0)$ and $tp(a,a_0)=tp(b,b_0)$. By Lemma \ref{orth}, $tp(a_i,a)=tp(b_i,b)$.

\item The point $a$ is greater than some point in $A$, say $a_1$, and the open ball $\mathfrak a:=B(a_1;a)$ contains no point of $A$.

Let $\mathcal A$ be the set of all open balls $B(a_1;a_i)$ for $a_i > a_1$. Let $n$ be the number of balls in $\mathcal A$ that have the same type $p$ as $\mathfrak a$. Then $tp(a_1)$ proves that there are at least $n+1$ open balls of type $p$ of center $a_1$. Therefore, $tp(b_1)$ proves the same thing. We can therefore find an open ball $\mathfrak b$ of center $b_1$ of type $p$ that contains no point from $B$. That ball contains a point $b$ such that $tp(b_1,b)=tp(a_1,a)$. Now, if $a_i$ is smaller than $a_1$, we have $tp(a_i,a_1)=tp(b_i,b_1)$ and $tp(a_1,a)=tp(b_1,b)$, therefore by Lemma \ref{orth}, $tp(a,a_i)=tp(b,b_i)$.

The fact that we have taken $b$ in a new open ball of center $b_1$ ensures that $B\cup \{b\}$ is again a sub-structure and that the two structures $A\cup \{a\}$ and $B\cup \{b\}$ are isomorphic.

\item The point $a$ is between two points of $A$, say $a_0$ and $a_1$ ($a_0 < a_1$), and there are no points of $A$ between $a_0$ and $a_1$.

Find a point $b$ such that $tp(a_0,a_1,a)=tp(b_0,b_1,b)$. Then if $i$ is such that $a_i > a$, we have $a_i \geq a_1$ and again by Lemma \ref{orth}, $tp(a_i,a) = tp(b_i,b)$. And same if $a_i < a$.

\end{enumerate}

\end{proof}

\begin{cor}\label{3types}
Let $A \subset T$ be any subset. Then $\bigcup_{(a,b,c) \in A^3} tp(a,b,c) \vdash tp(A)$.
\end{cor}
\begin{proof}
Let $A_0$ be the substructure generated by $A$. By the previous theorem the following set of formulas implies the type of $A_0$ :
\begin{itemize}
\item the quantifier-free type of $A_0$,
\item the set of 2-types $tp(a,b)$ for $(a,b) \in A_0^2$, $a < b$.
\end{itemize}
We need to show that those formulas are implied by the set of 3-types of elements of $A$. We may assume $A$ is finite.

First, the knowledge of all the 3-types is enough to construct the structure $A_0$. To see this, start of example with a point $a \in A$ maximal. Knowing the 3-types, one knows in what order the $b \wedge a, b\in A$ are placed. Doing this for all such $a$, enables one to reconstruct the tree $A_0$.

Now take $m_1 = a\wedge b$, $m_2=c \wedge d$ for $a,b,c,d \in A$ such that $m_1 \leq m_2$. The points $m_1$ and $m_2$ are both definable using only 3 of the points $a,b,c,d$, say $a,b,c$. Then $tp(a,b,c) \vdash tp(m_1,m_2)$.
\end{proof}

The previous results are also true, with the same proofs, for colored trees.
\\

It is proven in \cite{Par} that theories of trees are $NIP$. We give a more precise result.

\begin{prop}
Let $\mathcal T=(T,\leq,C_i)$ be a colored tree. Then $Th(\mathcal T)$ is dp-minimal.
\end{prop}
\begin{proof}

We will use criterium (5) of \ref{defs} : if $(a_i)_{i \in I}$ and $(b_j)_{j \in J}$ are mutually indiscernible sequences and $\alpha \in T$ is a point, then one of the sequences $(a_i)$ and $(b_j)$ is indiscernible over $\alpha$.

We will always assume that the index sets ($I$ and $J$) are dense linear orders without end points.
\\

\textbf{1)} We start by showing the result assuming the $a_i$ and $b_j$ are points (not tuples).

We classify the indiscernible sequence $(a_i)$ in 4 classes depending on its quantifier-free type.

\begin{description}
\item[ I] The sequence $(a_i)$ is monotonous (increasing or decreasing).
\item[II] The $a_i$ are pairwise incomparable and $a_i \wedge a_j$ is constant equal to some point $\beta$.
\item[III] The $a_i$ are incomparable and $a_i \wedge a_j$, $i < j$ depends only on $i$. Then let $a_i' = a_i \wedge a_j$ (for some $i < j$). The $a_i'$ form an increasing indiscernible sequence.
\item[IV] The $a_i$ are incomparable and $a_i \wedge a_j$, $i < j$ depends only on $j$. Then the $a_j' = a_i \wedge a_j$ ($i<j$) form a decreasing indiscernible sequence.
\end{description}

Assume $(a_i)$ lands in case \textbf{I}. Consider the set $\{x : x< \alpha\}$. If that set contains a non-trivial subset of the sequence $(a_i)$, we say that $\alpha$ \emph{cuts} the sequence. If this is not the case, then the sequence $(a_i)$ stays indiscernible over $\alpha$. To see this, assume for example that $(a_i)$ is increasing and that $\alpha$ is greater that all the $a_i$. Take two sets of indices $i_1 < … < i_n$ and $j_1 < … < j_n$ and a $k \in I$ greater that all those indices. Then $tp(a_{i_1},…,a_{i_n}/a_k) = tp(a_{j_1},…,a_{j_n}/a_k)$. Therefore by Lemma \ref{orth}, $tp(a_{i_1},…,a_{i_n}/\alpha) = tp(a_{j_1},…,a_{j_n}/\alpha)$.

In case \textbf{II}, note that if $(a_i)$ is not $\alpha$-indiscernible, then there is $i \in I$ such that $\alpha$ lies in the open ball $B(\beta;a_i)$ (we will also say that $\alpha$ \emph{cuts} the sequence $(a_i)$). This follows easily from Proposition \ref{typear}.

In the last two cases, if $(a_i)$ is $\alpha$-indiscernible, then it is also the case for $(a_i')$. Conversely, if $(a_i')$ is $\alpha$-indiscernible, then $\alpha$ does not cut the sequence $(a_i')$. From \ref{typear}, it follows easily that $(a_i)$ is also $\alpha$-indiscernible. We can therefore replace the sequence $(a_i)$ by $(a_i')$ which belongs to case \textbf{I}.
\\

Going back to the initial data, we may assume that $(a_i)$ and $(b_j)$ are in case \textbf{I} or \textbf{II}. It is then straightforward to check that $\alpha$ cannot cut both sequences. For example, assume $(a_i)$ is increasing and $(b_j)$ is in case \textbf{II}. Then define $\beta$ as $b_i \wedge b_j$ (any $i,j$). If $\alpha$ cuts $(b_j)$, then $\alpha > \beta$. But $(a_i)$ is $\beta$-indiscernible. So $\beta$ does not cut $(a_i)$. The only possibility for $\alpha$ to cut $(a_i)$ is that $\beta$ is smaller that all the $a_i$ and the $a_i$ lie in the same open ball of center $\beta$ as $\alpha$. But then the $a_i$ lie in the same open ball of center $\beta$ as one of the $b_j$. This contradicts mutual indiscernability.
\\

\textbf{2)} Reduction to the previous case. We show that if $(a_i)_{i\in I}$ is an indiscernible sequence of $n$-tuples and $\alpha \in T$ such that $(a_i)$ is not $\alpha$-indiscernible, then there is an indiscernible sequence $(d_i)_{i \in I}$ of points of $T$ in $dcl((a_i))$ such that $(d_i)$ is not $\alpha$-indiscernible.
\\

First, by \ref{3types}, we may assume that $n=2$. Write $a_i = (b_i,c_i)$ and define $m_i = b_i \wedge c_i$.

We again study different cases : 
\begin{enumerate}
\item The $m_i$ are all equal to some $m$.

As $(a_i)$ is not $\alpha$-indiscernible, necessarily, $\alpha > m$ and the ball $B(m;\alpha)$ contains one $b_i$ (resp. $c_i$). Then take $d_i=b_i$ (resp. $d_i=c_i$) for all $i$.

\item The $m_i$ are linearly ordered by $<$ and no $b_i$ nor $c_i$ is greater then all the $m_i$.

Then the balls $B(m_i; b_i)$ and $B(m_i; c_i)$ contain no other point from $(b_i,c_i,m_i)_{i \in I}$. Then, $\alpha$ must cut the sequence $(m_i)$ and one can take $d_i = m_i$ for all $i$.

\item The $m_i$ are linearly ordered by $<$ and, say, each $b_i$ is greater than all the $m_i$.

Then each ball $B(m_i;a_i)$ contains no other point from $(b_i,c_i,m_i)_{i \in I}$. If $\alpha$ cuts the sequence $m_i$, than again one can take $d_i =m_i$. Otherwise, take a point $\gamma$ larger than all the $m_i$ but smaller than all the $d_i$. Applying \ref{orth} with $a$ there replaced by $\gamma$, we see that $(b_i)$ cannot be $\alpha$-indiscernible. Then take $d_i = b_i$ for all $i$.

\item The $m_i$ are pairwise incomparable.

The the sequence $(m_i)$ lies in case \textbf{II}, \textbf{III} or \textbf{IV}. The open balls $B(m_i;b_i)$ and $B(m_i;c_i)$ cannot contain any other point from $(b_i,c_i,m_i)_{i\in I}$. Considering the different cases, one sees easily that taking $d_i = m_i$ will work.
\end{enumerate}

This finishes the proof.

\end{proof}

\begin{rem}
If we define dp-minimal$^+$ analogously to strongly$^+$-dependent (see \cite{Sh863}), all theories studied in this section are dp-minimal$^+$.
\end{rem}

\end{document}